\documentclass{conm-p-l}

\usepackage{mathrsfs,amssymb}
\usepackage[pdfauthor={Morten S. Risager and Zeév Rudnick},
pdftitle={On the statistics of the minimal solution of a linear
            Diophantine equation and uniform distribution of the real part of orbits in hyperbolic spaces}
pdfsubjects={Primary 11J71; Secondary },
pdfkeywords={We study a variant of a problem considered by Dinaburg
and Sinai on the statistics of the  minimal solution to a linear
Diophantine equation. We show that the signed ratio between  the
Euclidean norms of the minimal solution and the coefficient vector is
uniformly distributed modulo one.  We reduce the problem to an
equidistribution theorem of Anton Good concerning the orbits of a
point in the upper half-plane under the action of a Fuchsian group.},
            pdfproducer={LaTeX2e with hyperref},
            pdfcreator={Pdflatex},
            pdfdisplaydoctitle=true
]{hyperref}

\newtheorem{theorem}{Theorem}[section]
\newtheorem{lemma}[theorem]{Lemma}
\newtheorem{proposition}[theorem]{Proposition}
\theoremstyle{definition}

\theoremstyle{remark}
\newtheorem{remark}[theorem]{Remark}

\numberwithin{equation}{section}
\renewcommand{\Re}{\operatorname{Re}} 
\renewcommand{\Im}{\operatorname{Im}} 
\newcommand{\area}{\operatorname{area}}
\newcommand{\ecc}{\operatorname{sk}}

\def \G {{\Gamma}}
\def \g {{\gamma}}
\def \R {{\mathbb R}}
\def \H {{\mathbb H}}
\def \C {{\mathbb C}}
\def \Z {{\mathbb Z}}
\def \GinfmodG {{\Gamma_{\!\infty}\!\!\setminus\!\Gamma}}
\def \GmodH {{\Gamma\setminus\H}}
\def \vol {\hbox{vol}}
\def \sl  {\hbox{SL}_2(\mathbb Z)}
\def \slr  {\hbox{SL}_2(\mathbb R)}

\newcommand{\mattwo}[4]
{\left(\begin{array}{cc}
                        #1  & #2   \\
                        #3 &  #4
                          \end{array}\right) }

\newcommand{\inprod}[2]{\left \langle #1,#2 \right\rangle}

\begin{document}
\date{\today}
\thanks{The first author was funded by a Steno Research Grant  from The Danish Natural Science Research Council. The second author was supported by the Israel Science Foundation (grant No. 925/06).}
\title[statistics of minimal solutions and equidistribution]
{On the statistics of the minimal solution of a linear Diophantine
equation and uniform distribution of the real part of orbits in
hyperbolic spaces}    
\author{Morten S. Risager}
\address{Department of Mathematical Sciences, University of Aarhus, Ny Munkegade Building 530, 8000 Aarhus C, Denmark}
\email{risager@imf.au.dk}
\author{Ze\'ev Rudnick}
\address{School of Mathematical Sciences, Tel Aviv University, Tel Aviv 69978, Israel}
\email{rudnick@post.tau.ac.il}

\begin{abstract}
We study a variant of a problem considered by Dinaburg
and Sina\u\i{} on the statistics of the  minimal solution to a linear
Diophantine equation. We show that the signed ratio between  the
Euclidean norms of the minimal solution and the coefficient vector is
uniformly distributed modulo one.  We reduce the problem to an
equidistribution theorem of Anton Good concerning the orbits of a
point in the upper half-plane under the action of a Fuchsian group. 
\end{abstract}  
\maketitle

\section{Statement of results} \label{sec:statements}
\subsection{} 
For a pair of coprime integers $(a,b)$, the linear Diophantine
equation $ax-by=1$ is well known to have infinitely many integer
solutions $(x,y)$, any two differing by an integer multiple of
$(b,a)$. Dinaburg and Sina\u\i{}  \cite{DinaburgSinaui:1990a} studied the
statistics of the ``minimal'' such solution $v'=(x_0,y_0)$  
when the coefficient vector $v=(a,b)$ varies over
all primitive integer vectors lying in a large box with commensurate
sides.  
Their notion of ``minimality'' was in terms of the $L^\infty$-norm
$|v'|_\infty:=\max(|x_0|,|y_0|)$, and they studied the ratio 
$|v'|_\infty/|v|_\infty$, showing that it is uniformly distributed in
the unit interval.  
Other proofs were subsequently given by Fujii \cite{Fujii:1992a} who
reduced the problem to one about modular inverses, and then used
exponential sum methods, in particular a non-trivial bound on
Kloosterman sums, and by Dolgopyat \cite{Dolgopyat:1994a}, who used
continued fractions.

In this note, we consider a variant of the question by using
minimality with respect to the Euclidean norm $|(x,y)|^2:=x^2+y^2$ and 
study the ratio $|v'|/|v|$  
of the Euclidean norms as the coefficient vector varies over a large
ball. In this case too we find uniform distribution, in the interval
$[0,1/2]$. However, the methods involved appear quite different, as we
invoke an equidistribution theorem of Anton Good \cite{Good:1983a} 
which uses harmonic analysis on the modular curve.

\subsection{A lattice point problem} 
We recast the problem in slightly more general and geometric terms. 
Let $L\subset \C$ be a lattice in the plane, and let $\area(L)$ be the
area of a fundamental domain for $L$. Any primitive vector
$v$ in $L$  can be completed to a basis $\{v,v'\}$ of $L$. 
The vector $v'$ is unique up to a sign change and addition of a
multiple of $v$. In the case of the standard lattice $\Z[\sqrt{-1}]$,
taking $v=(a,b)$ and $v'=(x,y)$,  the condition that $v$, $v'$ 
give a basis of $\Z[\sqrt{-1}]$ is equivalent to requiring 
$ay-bx=\pm 1$.  
The question is: If we pick $v'$ to  minimize
the length $|v'|$ as we go through all possible completions, 
how does the ratio $|v'|/|v|$ between the lengths of $v'$ and $v$ fluctuate?
It is easy to see (and we will prove it below) that the ratio is
bounded, indeed that for a minimizer $v'$ we have  
$$
\frac{|v'|}{|v|} \leq \frac 12 +O(\frac 1{|v|^4})\;.
$$
We will show that the ratio $|v'|/|v|$ is uniformly distributed in
$[0,1/2]$ as  $v$ ranges over all primitive vectors of $L$ in a large
(Euclidean)  ball. 

We refine the problem  slightly by requiring 
that the lattice basis $\{v,v'\}$ is oriented positively, that is
$\Im(v'/v)>0$. Then $v'$ is unique up to addition of an integer
multiple of $v$. For the standard lattice $\Z[\sqrt{-1}]$ and
$v=(a,b)$, $v'=(x,y)$ the requirement is then that $ay-bx=+1$. 
Define  the {\em signed} ratio by
$$\rho(v):=\pm |v'|/|v|$$ 
where we chose $|v'|$ minimal, and the sign is $+$ if the angle
between $v$ and $v'$ is acute, and $-$ otherwise. 

\begin{theorem}\label{unif dist of rho}
As $v$ ranges over all primitive vectors in the
lattice $L$, the signed ratio $\rho(v)$ is uniformly
distributed modulo one.  
\end{theorem}
Explicitly, let $L_{prim}(T)$ be the set of primitive
vectors in $L$ of norm $|v|\leq T$. It is well known 
that  
\begin{equation*} \#L_{prim}(T) \sim \frac 1{\zeta(2)}
\frac{\pi}{\area(L)} T^2, \quad T\to \infty
\end{equation*}
Theorem~\ref{unif dist of rho} states that for any fixed subinterval
$[\alpha,\beta]\in (-1/2,1/2]$, 
$$ 
\frac 1{\#L_{prim}(T)} \{v\in L_{prim}(T): \alpha<\rho(v)<\beta \}
\to \beta-\alpha   
$$ 
as $T\to \infty$. 

\subsection{Equidistribution of real parts of orbits} 
We will reduce Theorem~\ref{unif dist of rho}  by geometric arguments
to a result of Anton Good \cite{Good:1983a} on uniform distribution of the
orbits of a point in the upper  half-plane under the action of a
Fuchsian group.   

Let $\G$ be discrete, co-finite,
non-cocompact subgroup of $\slr$. 
The group $\slr$ acts on the upper half-plane $\H=\{z\in \C:
\Im(z)>0\}$ by linear fractional transformations.  
We may assume, possibly after  conjugation in 
$\slr$, that $\infty$ is a cusp and that the stabilizer
$\G_{\!\infty}$ of $\infty$ in $\G$ is generated by
\begin{equation}
\pm \mattwo{1}{1}{0}{1}
\end{equation}
which as linear fractional transformation gives  
the unit translation $z\mapsto z+1$. (If $-I\notin \G$ there should be
no $\pm$ in front of the matrix). 
The group $\G=\sl$ is an example of such a group. 
We note that the
imaginary part of $\g(z)$ is fixed on the orbit $\G_{\!\infty}\g z$, and
that the real part modulo one is also fixed on this orbit. 
Good's theorem is 
\begin{theorem}[Good \cite{Good:1983a}]\label{equidistribution} 
Let $\G$ be as above and let $z\in\H$. 
Then $\Re(\G z)$ is uniformly distributed modulo one as $\Im(\g z)\to 0$. 
\end{theorem} 
More precisely, let 
$$
(\GinfmodG)_{\varepsilon,z}=\{\g\in\GinfmodG : \Im{\g z}>\varepsilon\}\;.
$$ 
Then for every continuous function $f\in C(\R\slash \Z)$, 
as $\varepsilon\to 0$,  
\begin{equation*}
\frac 1{ \#(\GinfmodG)_{\varepsilon,z}}
    \sum_{\g\in(\GinfmodG)_{\varepsilon,z}}f(\Re{\g z})
\to\int_{\R\slash\Z}f(t)dt \;. 
\end{equation*}

Though the writing in \cite{Good:1983a} is not easy to penetrate, the
results deserve to be more widely known. We sketch a
proof of Theorem~\ref{equidistribution} in
appendix~\ref{sec:spectral}, assuming familiarity with 
standard methods of the spectral theory of automorphic forms. 

\noindent{\bf Acknowledgements:} We thank Peter Sarnak for his comments
on an earlier version and for alerting us to Good's work.

\section{A geometric argument}\label{sec:Geom}  
\subsection{} 
We start with a basis $\{v,v'\}$ for the lattice $L$ 
which is oriented positively, that is
$\Im(v'/v)>0$. For a given $v$,  $v'$ is unique up to addition of an integer
multiple of $v$. Consider the parallelogram $P(v,v')$ spanned by $v$ and
$v'$. Since $\{v,v'\}$ form a basis of the lattice $L$, $P(v,v')$ is a
fundamental domain for the lattice and the area of $P(v,v')$ depends
only on $L$, not on $v$ and $v'$: $\area(P(v,v'))=\area(L)$.

Let $\mu(L)>0$ be the minimal length of a nonzero vector in $L$:
$$ 
\mu(L)=\min\{ |v|:0\neq v\in L\}\;.
$$

\begin{lemma}
Any minimal vector $v'$ satisfies
\begin{equation}\label{upper bd on v'}
|v'|^2\leq (\frac{|v|}2 )^2 + (\frac {\area(L)} {|v|})^2 \;.
\end{equation}
Moreover, if $|v|>2\area(L)/\mu(L)$ then the minimal
vector $v'$ is unique up to sign. 
\end{lemma}
\begin{proof}
To see \eqref{upper bd on v'}, note that the height of the
parallelogram $P$ spanned by $v$ and $v'$ is $\area(P)/|v| =
\area(L)/|v|$. If $h$ is the 
height vector, then the vector $v'$ 
thus lies on the affine line $h+\R v$ so is of the form 
$h+tv$. After adding an integer multiple of $v$ we may assume
that $|t|\leq 1/2$, a choice that minimizes $|v'|$, and then 
$$ |v'|^2 = t^2|v|^2+ |h|^2\leq \frac 14 |v|^2 +
(\frac{\area(L)}{|v|})^2  \;.
$$

We now show  that for $|v|\gg_L 1$, the minimal choice of $v'$ is unique if
we assume $\Im(v'/v)>0$, and up to sign otherwise: 
Indeed, writing the minimal $v'$ as above in the form 
$v'=h+tv$ with $|t|\leq 1/2$,  the choice of $t$ is unique unless we
can take $t=1/2$, in which case we have the two choices $v'=h\pm
v/2$. To see that $t=\pm 1/2$ cannot occur for $|v|$ sufficiently
large, we argue that if $v'=h+v/2$ then we must have $2 h=2v'-v\in
L$. 
The length of the nonzero vector $2h$ must then be
at least $\mu(L)$. Since $|h|=\area(L)/|v|$ this gives
$2\area(L)/|v|\geq \mu(L)$, that is 
$$|v|\leq \frac{2\area(L)}{\mu(L)}$$
Hence $v'$ is uniquely determined if $|v|>2\area(L)/\mu(L)$. 
\end{proof}

\subsection{}

Let $\alpha=\alpha_{v,v'}$ be the angle between $v$
and $v'$,  which takes values between $0$ and $\pi$ since
$\Im(v'/v)>0$. As is easily seen, for any choice of $v'$,
$\sin\alpha_{v,v'}$  shrinks  
as we increase $|v|$, in fact we have: 
\begin{lemma}\label{lem:angle}
For any choice of $v'$ we have 
\begin{equation}\label{upper bd on alpha} 
\sin \alpha \leq \frac{\area(L)}{\mu(L)}\frac 1{|v|}  \;.
\end{equation}
\end{lemma}
\begin{proof}
To see \eqref{upper bd on alpha}, note that the area of the
fundamental parallelogram $P(v,v')$ is  given in terms 
of $\alpha$ and the side lengths by 
$$  \area(P) =|v| |v'|\sin \alpha  $$
and since $v'$ is a non-zero vector of $L$, we necessarily have
$|v'|\geq \mu(L)$ 
and hence, since $\area(P)=\area(L)$ is independent of $v$, 
$$
0<\sin \alpha \leq \frac{\area(L)}{\mu(L)|v|} 
$$
as claimed. 
\end{proof}
Note that if we take for $v'$ with minimal length, then we have a lower
bound $\sin\alpha \geq 2\area(L)/|v|^2 +O( 1/|v|^6)$
obtained by inserting \eqref{upper bd on v'} into the
area formula $\area(L)=|v||v'|\sin\alpha$.

\subsection{} 
Given a positive basis $\{v,v' \}$, we define a measure of 
skewness of the fundamental parallelogram as follows: 
Let $\Pi_v(v')$  be the orthogonal projection of the vector $v'$ to
the line through  $v$. It is a scalar multiple of $v$: 
$$\Pi_v(v')= \ecc(v,v') v$$
where the multiplier $\ecc(v,v')$, which we call the {\em
skewness} of the parallelogram, is given in terms of the inner
product between $v$ and $v'$  as 
\begin{equation}\label{exp ecc}
\ecc(v,v') =  \frac{\langle v',v\rangle}{|v|^2}  \;.
\end{equation}
Thus we see that the  skewness is the real part of the ratio 
$v'/v$: 
\begin{equation*}
\ecc(v,v') = \Re(v'/v) \;.
\end{equation*}

If we replace $v'$ by adding to it an integer multiple of $v$, then
$\ecc(v,v')$  changes by 
$$\ecc(v,v'+nv) = \ecc(v,v') + n   \;.
$$
In particular, since $v'$ is unique up to addition of an integer
multiple of $v$, looking at the fractional part, that is in $\R/\Z$,
we get a quantity $\ecc (v)\in (-1/2,1/2]$ depending only on $v$: 
$$
   \ecc(v) : =\ecc(v,v') \mod 1 \;.
$$
This is the least skewness of a fundamental domain for the lattice 
constructed from the primitive vector $v$. 

\begin{lemma}
  The signed ratio $\rho(v) = \pm |v'|/|v|$ and the least skewness 
  $\ecc(v)$ are asymptotically equivalent: $$\rho(v) =
  \ecc(v)\left(1+O(\frac 1{|v|^2})\right) \;.
$$
\end{lemma}
\begin{proof}
In terms of the angle $0<\alpha<\pi$ between the vectors $v$ and $v'$,
we have  
\begin{equation*}\label{relation between ecc and alpha} 
\ecc(v,v') = \frac{|v'|}{|v|}\cos\alpha  \;.
\end{equation*} 
Our claim follows from this and 
the fact  $\cos\alpha = \pm 1+O(1/|v|^2)$,  
which follows from the upper bound \eqref{upper bd on alpha} of
Lemma~\ref{lem:angle}.   
\end{proof}

Thus the sequences $\{\rho(v)\}$, $\{\ecc(v)\}$ are asymptotically
identical, hence uniform distribution of one implies that of the other. 
To prove Theorem~\ref{unif dist of rho}  it suffices to show 
\begin{theorem}\label{unif dist of ecc}
As $v$ ranges over all primitive vectors in the
lattice $L$, the least skewness $\ecc(v)$ become uniformly
distributed modulo one.  
\end{theorem}
This result, for the standard lattice $\Z[\sqrt{-1}]$,  
was highlighted by Good in the introduction to \cite{Good:1983a}. 
Below we review the reduction of Theorem~\ref{unif dist of ecc} 
to Theorem~\ref{equidistribution}. 

\subsection{Proof of Theorem~\ref{unif dist of ecc}} 
Our problems only depend on the lattice $L$ up to scaling. So we may
assume that $L$ has a basis $L=\{1,z\}$ with $z=x+iy$ in the upper
half-plane. The area of a fundamental domain for $L$ is
$\area(L)=\Im(z)$. 
Any primitive vector has the form $v=cz+d$ with the
integers $(c,d)$ co-prime.

Now given the positive lattice basis $v=cz+d$ and $v'=az+b$, form the
integer matrix  
$\g=\begin{pmatrix} a&b\\ c&d\end{pmatrix}$  , which has $\det(\g)=+1$
since $\{v,v'\}$ form a positive basis of the lattice. Thus we get a
matrix in the modular group $\G=SL_2(\Z)$. 
Then with $\g$ applied as a M\"{o}bius transformation to $z$, 
the length of $v$ can be computed via 
\begin{equation} \Im(\g z) = \frac{\Im(z)}{|cz+d|^2}=\frac{\area(L)}{|v|^2}
\end{equation}
The signed ratio between the lengths of $v$ and $v'$ (when $v'$ is chosen of minimal length) is 
\begin{equation}\rho{(v)} = \pm |\gamma z| \;.\end{equation}
 where the sign is $+$ if $\Re(\g z)>0$ and $-$ otherwise.
Moreover, we have 
 $$\ecc(v,v') = \Re(\g z)  $$
Indeed, 
$$
\Re (\g z) = \frac{ac(x^2+y^2) +(ad+bc)x +bd}{|cz+d|^2}
$$ 
which is $\ecc(v,v')$ in view of \eqref{exp ecc}. 
Consequently, the uniform distribution modulo one of $\ecc(v)$ as
$|v|\to\infty$ is then exactly the
uniform distribution modulo one of $\Re(\g z)$ as $\g$ varies over
$\GinfmodG$ with $\Im(\g z )\to 0$, that is 
Theorem~\ref{equidistribution}.  
\qed

\appendix
\section{A sketch of a proof of Good's theorem}\label{sec:spectral}

To prove Theorem~\ref{equidistribution}, we use Weyl's criterion to
reduce it to showing that the corresponding ``Weyl sums'' satisfy
\begin{equation}\label{character asymptotics} 
  \sum_{\g \in(\GinfmodG)_{\varepsilon,z}}e(m\Re{\g
z})=\delta_{m=0}\frac{t_\G}{\vol{(\GmodH)}}\frac 1\varepsilon 
+o(1/\varepsilon)
\end{equation} 
as $\varepsilon\to 0$. 
Here $t_\G$ equals $2$ if $-I\in \G$ and $1$ otherwise. 
In turn, \eqref{character asymptotics} will follow, by a more or less
standard Tauberian theorem  
(see e.g. \cite[p. 1035-1038]{PetridisRisager:2004a}) from knowing the
analytic properties  of the series
\begin{equation*}
  V_m(z,s):=\sum_{\g\in\GinfmodG}\Im(\g z)^se(m\Re(\g z)) \;.
\end{equation*}
studied also in \cite{Good:1981b,Neunhoffer:1973a}
Here $e(x)=\exp(2\pi i x)$. The series is absolutely convergent for
$\Re(s)>1$, as is seen by comparison with  the
standard non-holomorphic Eisenstein series $V_0(z,s)=E(z,s)$ of weight
$0$ (See \cite{Selberg:1989a}). 
For general $m$ the series is closely related
to the Poincar\'e series 
$$U_m(z,s) = \sum_{\g\in\GinfmodG}\Im(\g z)^s e(m\g z)$$ 
studied by Selberg \cite{Selberg:1965a}. For a different
application of the series $V_m(z,s)$, see \cite{Sarnak:2001a}. 



The analytic properties from which we can conclude
Theorem~\ref{equidistribution} are given by  
\begin{proposition}\label{continuation} 
The series $V_m(z,s)$ admits
meromorphic continuation to $\Re(s)>1/2.$ If poles exist they are real
and simple. If $m\neq 0$ then $V_m(z,s)$ is regular at $s=1$. If $m=0$ the
point $s=1$ is a pole with residue $t_\G/\vol{(\GmodH)}$. 
Moreover,  $V_m(z,s)$ has polynomial growth on vertical strips in
$\Re(s)>1/2$.  
\end{proposition}
\noindent{\bf Sketch of proof.} 
  The claim about continuation of $V_0(z,s)=E(z,s)$ is well-known and
goes back to Roelcke  \cite{Roelcke:1956a} and Selberg
\cite{Selberg:1963a}. To handle also  $m\neq 0$ we may adopt the
argument of Colin de Verdi\`ere  \cite[Th\' eor\`eme
3]{Colin-de-Verdiere:1983a} and of Goldfeld and Sarnak \cite{GoldfeldSarnak:1983a} 
to get the result. This is done as
follows: Consider  the hyperbolic Laplacian 
\begin{equation*}
\Delta=-y^2\left(\frac{\partial^2 }{\partial x^2}+
\frac{\partial^2 }{\partial y^2}\right) \;.
  \end{equation*}
If we restrict $\Delta$  to smooth functions on $\GmodH$ which are compactly supported it defines an essentially self-adjoint
operator on $L^2(\GmodH,d\mu)$ where $d\mu(z)=dxdy/y^2$, with inner
product 
\newcommand{\Laplace}{\Delta} 
\begin{equation*}
  \inprod{f}{g}=\int_{\GmodH}f(z)\overline{g(z)}d\mu(z).
\end{equation*}
 We will also denote by $\Laplace$ the self-adjoint closure. 

Let $h(y)$ be a smooth function which equals 0 if $y<T$ and 1 if
$y>T+1$ where $T$ is sufficiently large. One may check that when $\Re(s)>1$
 \begin{equation*}
   V_m(z,s)-h(y)y^se(mx)
 \end{equation*}
is square integrable. This is an easy exercise using \cite[Theorem
2.1.2]{Kubota:1973a}. The series $V_m(z,s)$ satisfies 
\begin{equation}\label{straightforward}
  (\Delta-s(1-s))V_m(z,s)=(2\pi m)^2V_m(z,s+2) \textrm{ when } \Re(s)>1,
\end{equation}
since $f_s(z)=y^s e^{2\pi i m \Re z}$ satisfies this equation and
because the Laplacian commutes with isometries, so does $V_m(z,s)$, being a sum
of translates of $f_s$. 
Therefore
\begin{align}
 \nonumber (\Delta-s(1-s))&(V_m(z,s)-h(y)y^se(mx))\\ =(2\pi m)^2& (V_m(z,s+2)-h(y)y^{s+2}e(mx))\\ \nonumber &-h''(y)y^{s+2}e(mx)-2h'(y)y^{s+1}e(mx)
\end{align}
is also square integrable, since the last two terms are compactly
supported. We can therefore use the resolvent $(\Laplace-s(1-s))^{-1}$ to
invert this and find 
\begin{equation*}
V_m(z,s)- h(y)y^se(mx)=(\Laplace-s(1-s))^{-1}((2\pi m)^2V_m(z,s+2)-H(z,s))
\end{equation*} 
where \begin{equation*}
H(z,s)=(2\pi
m)^2h(y)y^{s+2}e(mx))+h''(y)y^{s+2}e(mx)+2h'(y)y^{s+1}e(mx)
\end{equation*}
This defines the meromorphic continuation of $V_m(z,s)$ to $\Re(s)>1/2$ by
the meromorphicity of the resolvent (see e.g
\cite{Faddeev:1967a}). The singular points are simple and contained
in the set  of $s\in \C$ such that $s(1-s)$ is an eigenvalue of
$\Laplace$. Since $\Laplace$ is self-adjoint, these lie on the real line (when
$\Re(s)>1/2$). The potential pole at $s=1$ has residue a constant
times 
\begin{equation*}
 \int_{\GmodH}(2\pi m)^2V_m(z,3)-H(z,1)d\mu
\end{equation*}
The contribution from $h''(y)y^{s+2}e(mx)+2h'(y)y^{s+1}e(mx)$ is
easily seen to be zero if $T$ is large enough using $\int_0^1
e(mx)dx=0$ when $m\neq 0$. To handle the rest we may unfold to get  
\begin{align*}
 (2\pi m)^2&\int_{\GmodH}(V_m(z,3)-h(y)y^3e(mx))d\mu(z)\\& = (2\pi
m)^2\int_0^\infty\int_0^1(y^3-h(y)y^3)e(mx)y^{-2}dxdy=0 
\end{align*}
so $V_m(z,s)$ is analytic at $s=1$. The claim about growth in vertical
strips is proved as in \cite[Lemma 3.1]{PetridisRisager:2004a}. \qed

\begin{remark}
  It is possible to extend the main idea of the proof of Proposition
\ref{continuation} to prove the meromorphic continuation of $V_m(z,s)$ to
$s\in \C$. 
But since our main aim was to prove Theorem \ref{equidistribution} we
shall stop here.  
\end{remark}

\bibliographystyle{plain}
\def\cprime{$'$}

\end{document}